\documentclass[11pt]{article}
\usepackage{amssymb}
\usepackage{amsmath}
\usepackage{amscd}
\usepackage{amsbsy}
\usepackage{amsfonts}
\usepackage{theorem}
\usepackage{hyperref}

\input xy
\xyoption{all}

\newcommand{\zz}{{\Bbb Z}}

\newcommand{\pp}{{\Bbb P}}

\newcommand{\ff}{{\Bbb F}}

\newcommand{\ddeg}{\operatorname{deg}}

\newcommand{\kker}{\operatorname{Ker}}

\newcommand{\op}[1]{\operatorname{#1}}

\newcommand{\ffi}{\varphi}

\newcommand{\eps}{\varepsilon}

\newcommand{\la}{\langle}
\newcommand{\ra}{\rangle}
\newcommand{\lva}{\langle\!\langle}   
\newcommand{\rva}{\rangle\!\rangle}   
\newcommand{\row}{\rightarrow}
\newcommand{\llow}{\longleftarrow}

\newcommand{\lrow}{\longrightarrow}

\renewcommand{\geq}{\geqslant}

\newcommand{\nichego}[1]{}

\newcommand{\wt}[1]{\widetilde{#1}}

\newcommand{\dmkF}[1]{\op{DM}(k;#1)}

\newcommand{\CH}{\op{CH}}

\newcommand{\Qed}{\hfill$\square$\smallskip}

\newenvironment{proof}{\noindent{\it Proof}:}{\vskip 5mm}

\newtheorem{prop}{Proposition}[section]{\bf}{\it}
\newtheorem{thm}[prop]{Theorem}{\bf}{\it}
{\bf}{\it}
{\bf}{\it}
\newtheorem{defi}[prop]{Definition}{\bf}{\it}
{\bf}{\it}
{\bf}{\it}
\newtheorem{exa}[prop]{Example}{\bf}{\it}
\newtheorem{rem}[prop]{Remark}{\bf}{}
{\bf}{\it}

{\bf}{\it}
\newtheorem{claim}[prop]{Claim}{\bf}{\it}
\newtheorem{cor}[prop]{Corollary}{\bf}{\it}
{\bf}{\it}

\begin{document}

\title{Varieties and pure symbols}
\author{Alexander Vishik\footnote{School of Mathematical Sciences, University
of Nottingham}}
\date{}
\maketitle

\begin{abstract}
In this article, we prove that the $2$-isotropy of any projective variety is controlled by a pure symbol in $K^M_*/2$ over the flexible closure of the base field. We also show that such pure symbols
control the $2$-equivalence of field extensions as well as the numerical equivalence of algebraic cycles (with $\ff_2$-coefficients). 
\end{abstract}

\section{Introduction}

The essential ingredient of the proof of all cases of Milnor and Bloch-Kato conjectures starting from the foundational 
work of Merkurjev \cite{MeNRS2} to the final solution by 
Voevodsky and Rost \cite{VoMil},\cite{VoZl},\cite{RoNVAC} was the use of norm-varieties for pure symbols in Milnor's K-theory mod $p$. Such a
variety controls the triviality of a particular symbol
and the passage from the base field to the generic point of it annihilates the symbol in the gentlest possible way.
Thus, one may reduce the problem to the easy case where Milnor's K-theory is $p$-divisible, if one has enough control over the above passage. 
For $p=2$, Pfister quadrics are norm-varieties. For odd primes, such varieties were constructed by Rost \cite{RoNVAC}. The main property of these is that, over all extensions $L/k$ of the base field, the $p$-isotropy of the variety over $L$ is equivalent to the triviality
of the symbol restricted to $L$.

It appears that an arbitrary projective variety is a norm-variety for an appropriate pure symbol in $K^M_*/2$. The only thing, this symbol is defined not over the ground field, but over some purely transcendental extension of it. This is our main result - Theorem \ref{Main}.

\begin{thm}
 Let $k$ be a field of $char\neq 2$, and $X$ be some projective variety over $k$. Then there exists a purely
 transcendental extension $k(A)/k$ and a pure symbol
 $\alpha\in K^M_*(k(A))/2$, such that, for any $L/k$,
 $$
 \alpha|_{L(A)}=0\hspace{3mm}\Leftrightarrow\hspace{3mm}
 X_L\,\,\text{is}\,\,\,2-\text{isotropic}.
 $$
\end{thm}

Thus the $2$-isotropy of any given projective variety $X$ is controlled by a certain pure symbol over the {\it flexible closure} of the base field.

This implies that such symbols also control the 2-equivalence of field extensions - Proposition \ref{two-equiv-ker}. Observe, that the $2$-equivalence classes of field extension (for $char(k)=0$)
describe {\it isotropic points} of characteristic $2$ of the Balmer spectrum \cite{Bal-1} of Voevodsky motives - see \cite[Theorem 5.13]{INCHKm}. The above permits to show that all these points are
closed - see \cite[Theorem 3.7]{BSPS}.
Finally, combining our main theorem with the main result of \cite{INCHKm}, we obtain that the numerical triviality of algebraic cycles is also controlled by pure symbols over the flexible closure of $k$ - see Theorem \ref{num-triv-symb} and Corollary \ref{Chow-num-triv-symb}.

We deduce our main result from the Theorem of Colliot-Th\'el\`ene and Levine - Theorem \ref{CT-L-thm} claiming that the isotropy of the generic representative of a linear system of quadrics is equivalent to the isotropy of the base set of this linear system. This permits to reduce the case of an arbitrary projective variety to that of a smooth quadric. The case of a quadric was known from \cite{kerM},
where it was proven using the algebraic theory of quadratic forms. We, instead, deduce it from the
same Theorem of Colliot-Th\'el\`ene and Levine, by deducing the Pfister Representation Theorem - see Theorem \ref{PRT}, which was the main tool in the algebraic arguments of \cite{kerM}.

The paper is organised as follows. In Section \ref{two} we prove our Main Theorem. To make the arguments self-contained, we provide a complete proof here, including an alternative proof of
Theorem \ref{CT-L-thm}. In Section \ref{three} we deduce the Pfister Representation Theorem from
that of Colliot-Th\'el\`ene and Levine. Finally, in Section \ref{four} we discuss the applications to the 2-equivalence of field extensions and the numerical equivalence of cycles.

\section{The main theorem}
\label{two}

\begin{defi}
 Let $\alpha\in K^M_n(k)/p$ and $X$ - a variety over $k$.
 We call $X$ a {\it norm-variety} for $\alpha$, if for any 
 field extension $L/k$, 
 $
 \alpha|_L=0\hspace{3mm}\Leftrightarrow\hspace{3mm}
 X_L\,\,\text{is}\,\,\,p-\text{isotropic}.
 $
\end{defi}

The norm-varieties are known to exist for pure symbols 
(for any $p$ and $n$), for sums of two symbols mod $2$,
for $n=2$ (and arbitrary $p$ and $\alpha$), and in few 
other cases. The existence of such varieties has useful consequences for Milnor's K-theory/Galois cohomology.

\begin{exa}
 Let $char(k)\neq 2$ and $\alpha=\{a,b\}\in K^M_*(k)/2$
 be a pure symbol of $\ddeg=2$ (mod $2$). Then the conic
 $C_{\alpha}$ defined by the form $\la 1,-a,-b\ra$ is a
 norm-variety for $\alpha$.
\end{exa}

While norm-varieties control the triviality of the respective elements of Milnor's K-theory mod $p$, 
the elements, in turn, control the $p$-isotropy of
these varieties.
The purpose of this paper is to show that, in fact, the $2$-isotropy of an arbitrary projective variety $X/k$ is
controlled by a certain pure symbol in Milnor's K-theory
mod $2$ of the {\it flexible closure} $\wt{k}=k(\pp^{\infty})$ of $k$. So, in a sense, any projective
variety is a {\it norm-variety}.

\begin{thm}
 \label{Main}
 Let $k$ be a field of $char\neq 2$, and $X$ be some projective variety over $k$. Then there exists a purely
 transcendental extension $k(A)/k$ and a pure symbol
 $\alpha\in K^M_*(k(A))/2$, such that, for any $L/k$,
 $$
 \alpha|_{L(A)}=0\hspace{3mm}\Leftrightarrow\hspace{3mm}
 X_L\,\,\text{is}\,\,\,2-\text{isotropic}.
 $$
\end{thm}

\begin{proof}
 We know that Pfister quadrics are norm-varieties for the respective pure symbols $\alpha=\{a_1,\ldots,a_n\}\in K^M_n(k)/2$. We will reduce to the case of a Pfister quadric in two steps.\\
 
 \noindent
 {\bf \underline{Step 1}:}
 Here we reduce an arbitrary projective variety to a smooth quadric.
 
 \begin{claim}
  \label{cl-first}
  In the assumptions of Theorem \ref{Main}, there exists a purely transcendental extension $k(B)/k$ and a smooth
  projective quadric $Q$ over $k(B)$, such that, for any
  extension $L/k$, $X_L$ is $2$-isotropic\hspace{2mm}
  $\Leftrightarrow$\hspace{2mm} $Q_{L(B)}$ is $2$-isotropic.
 \end{claim}

 \begin{proof}
  Embed $X$ into some projective space $\pp^N$. Applying a Veronese embedding, if needed, we may assume that $X$ is defined in the projective space by quadrics. In other words, $X$ is the base set of some linear system
  of quadrics. 
  
  Now we may use the following result of Colliot-Th\'el\`ene and Levine.
  
  \begin{thm} {\rm (\cite[Theorem 3]{C-TL})}
   \label{CT-L-thm}
   The generic representative of a linear system of quadrics is $2$-isotropic if and only if the base set of it is $2$-isotropic.   
  \end{thm}

  To make the arguments self-contained and transparent, I will provide an alternative proof here (which follows the proof of \cite[Statement 3.1]{Iso}). 
  
  \begin{proof}
   Let $X\subset\pp^N$ be the base set of the linear system of quadrics parametrized by some $\pp^M$.
   Consider $Y=\{(u,H)\subset (\pp^N\backslash X)\times\pp^M\,|\, u\in H\}$. It has natural projections:
   $$
   \pp^N\backslash X\stackrel{\ffi}{\llow} Y
   \stackrel{\psi}{\lrow}\pp^M,
   $$
   where $\ffi$ is a $\pp^{M-1}$-bundle.
   Let $\eta\hookrightarrow\pp^M$ be the generic point and $Y_{\eta}$ be the generic fiber of $\psi$. Then
   $Y_{\eta}=Q_{\eta}\backslash X$, where $Q_{\eta}$
   is the generic representative of our linear system of quadrics. Then, by the Projective Bundle Theorem,
   $\CH^*(Y)=\CH^*(\pp^N\backslash X)[\rho]/(\text{some relation})$, where $\rho=c_1(O(1))$ is the first Chern class of the canonical line bundle of our projective bundle. Note that this canonical bundle $O(1)$ is the 
   restriction of the bundle $O(1)$ from $\pp^M$, so it is trivial when restricted to the generic fiber. We 
   have the natural surjection on Chow groups:
   $\CH^*(Y)\twoheadrightarrow\CH^*(Y_{\eta})$ which maps
   $\rho$ to zero. Thus, we obtain the natural surjection
   $\CH_1(\pp^N\backslash X)\stackrel{f}{\twoheadrightarrow}\CH_0(Y_{\eta})=\CH_0(Q_{\eta}\backslash X)$ (the pull back via $\ffi$ composed
   with the restriction to the generic fiber). 
   Note that the source-group here is cyclic generated
   by the class $l_1$ of a projective line on $\pp^N$.
   If $X$ is isotropic, then so is $Q_{\eta}$ (as it contains the former). If $X$ is anisotropic, then
   the degree map $\ddeg:\CH_0(Q_{\eta}\backslash X)\row\zz/2$ is well-defined and $\ddeg(f(l_1))=\ddeg(l_1\cdot Q)=2=0\in\zz/2$ (where $Q$ is any representative of our linear system). Hence, $Q_{\eta}\backslash X$ is anisotropic $\Rightarrow$ so is $Q_{\eta}$.
   \Qed
  \end{proof}
  
  It remains to set $k(B)=k(\pp^M)$ and $Q=Q_{\eta}$.
  Note that, if $X$ is anisotropic, then $Q$ is smooth, as singular quadrics are isotropic (and if $X$ is isotropic, we may choose any smooth isotropic quadric).
  The Claim is proven.
  \Qed
 \end{proof}
 
 \noindent
 {\bf \underline{Step 2}:}
 Here we reduce an arbitrary smooth quadric to a Pfister one.
 
 \begin{claim}
  \label{cl-second}
  Let $char(F)\neq 2$ and $Q$ be a smooth projective quadric over $F$. Then there exists a purely transcendental extension $F(C)/F$ and a pure symbol
  $\alpha\in K^M_*(F(C))/2$, such that, for any field extension $L/F$, $Q_L$ is $2$-isotropic
  \hspace{2mm}$\Leftrightarrow$\hspace{2mm}
  $\alpha_{L(C)}=0$.  
 \end{claim}

 This is \cite[Statement 2]{kerM}. Again, to make the arguments self-contained and to demonstrate the relation between the methods of Step 1 and Step 2, I will provide the proof.
 
 \begin{proof}
 WLOG we may assume that $Q$ is given by the quadratic
 form $q=\la 1,-a_1,-a_2,\ldots,-a_r\ra$. Let 
 $q_i=\la 1,-a_1,\ldots,-a_i\ra$. We construct a tower
 of purely transcendental extensions
 $$
 F\row F(C_1)\row F(C_2)\row\ldots\row F(C_r)
 $$ 
 with
 pure symbols $\beta_i\in K^M_*(F(C_i))/2$, such that
 $q_i|_L$ is isotropic $\Leftrightarrow$ $\beta_i|_{L(C_i)}=0$.
 
 Start with $F(C_2)=F$ and $\beta_2=\{a_1,a_2\}$ (I leave the case $i=1$ to the reader). It is well-known that the isotropy of $q_2=\la 1,-a_1,-a_2\ra$ is equivalent to that of $q_{\beta_2}=\lva a_1,a_2\rva$ (since the former is a Pfister neighbour of the latter). Also, by construction, $q_2\subset q_{\beta_2}$. It remains to do the step.
 
 {\bf step:} We have embeddings of forms:
 $q_i\subset q_{i+1}$ and $q_i|_{F(C_i)}\subset q_{\beta_i}$. Let $s$ be the orthogonal complement to
 the latter embedding and $p=s\perp\la a_{i+1}\ra$.
 Then $q_{i+1}|_{F(C_i)}=q_{\beta_i}-p$. 
 Let $L/F$ be any field extension. 
 Either $q_{\beta_i}|_{L(C_i)}$ is isotropic, which implies (by inductive assumption) that $q_i|_L$ is isotropic, and so is $q_{i+1}|_L$ (containing it),
 or $q_{\beta_i}|_{L(C_i)}$ is anisotropic. In the latter case, $q_{i+1}|_L$ is isotropic $\Leftrightarrow$ $q_{i+1}|_{L(C_i)}$ is isotropic (as extension is purely transcendental) $\Leftrightarrow$ $p_{L(C_i)}\subset q_{\beta_i}|_{L(C_i)}$. By the Pfister Representation theorem, this is equivalent to: $q_{\beta_i}\perp -\la p(X)\ra$ is isotropic over
 $L(C_i)(V_p)$ (the function field of the underlying vector space of the form $p$), where $p(X)$ is the {\it generic value} of the quadratic form $p$. Note that the latter form is a (minimal) neighbour of the Pfister form
 $q_{\beta_i\cdot\{p(X)\}}$ and so, their isotropies are
 equivalent. Set $F(C_{i+1})=F(C_i)(V_p)$
 and $\beta_{i+1}=\beta_i\cdot\{p(X)\}$. Then, we get:
 $q_{i+1}|_L$ is isotropic $\Leftrightarrow$ $q_{\beta_{i+1}}|_{L(C_{i+1})}$ is isotropic.
 Note also that, by construction, $q_{i+1}|_{F(C_{i+1})}$ is a subform of $q_{\beta_{i+1}}$ (for example, because the former represents $1$ and the latter becomes split
 over its function field). The induction step is proven.
 
 It remains to set $F(C)=F(C_r)$ and $\alpha=\beta_r$.
 The Claim is proven.
  \Qed
 \end{proof}
 
 Combining Claims \ref{cl-first} and \ref{cl-second}, we
 get a purely transcendental extension $k(A)=k(B)(C)$ and a pure symbol $\alpha\in K^M_*(k(A))/2$, such that, for any field extension $L/k$,
 $$
 X_L\text{ is }2-\text{isotropic }\Leftrightarrow
 \,\,\,\alpha_{L(A)}=0.
 $$
 Theorem \ref{Main} is proven.
 \Qed
\end{proof}

\begin{rem}
 Note that there are no restrictions on the projective variety $X$ in the Theorem, in particular, it may be singular, reducible, or disconnected.
\end{rem}

At the first glance, the methods used in both steps are quite different: while Step 1 is done with the help of algebraic cycles, Step 2 uses Algebraic Theory of Quadratic Forms and the Representation Theorem
of Pfister as the main tool. But, in reality, Step 2 can be done using the methods of Step 1, that is, the Theorem of Colliot-Th\'el\`ene and Levine. Moreover,
the Pfister Theorem itself can be obtained this way - this will be done in the next section.

\begin{rem}
 The analogue of the Main Theorem should also hold for odd primes.    
\end{rem}

\section{The Representation Theorem of Pfister and algebraic cycles}
 \label{three}

The following theorem was proven by Pfister as a consequence of the Cassels-Pfister theorem.

\begin{thm} {\rm (\cite{Pf})}
 \label{PRT}
 Let $char(k)\neq 2$. An anisotropic form $q$ contains the form $p$ as a subform if and only if it represents
 the generic value of it.
\end{thm}

Let us deduce it from the Theorem of Colliot-Th\'el\`ene and Levine. \\

\begin{proof}
 In one direction, the statement is obvious. Suppose, $q$ represents the generic value of $p$. Then it represents
 also the generic value of any subform $p'$ of it. Indeed, such a value is a special value of $p$ over the function field of $V_{p'}$. We may consider the quadratic form $q\perp -\la p(x)\ra$ giving a quadric
 fibration over $V_p$. The generic fiber of it is isotropic. But, for any projective map $Z\row\op{Spec}(S)$ over the spectrum of a DVR with generic and special
 fibers $Z_{\eta}$ and $Z_{\eps}$, we have a specialisation map $\CH_0(Z_{\eta})\row\CH_0(Z_{\eps})$
 which respects the degree. So, all the special fibers
 are isotropic as well. Hence, the generic value of $p'$
 is also represented.
 
 Using induction on subforms, we may assume that a subform $p'\subset p$ of co-dimension $1$ is also a subform of $q$.
 Let $p'=\la a_1,\ldots,a_{n-1}\ra$, $p=p'\perp\la a_n\ra$
 and $q=\la a_1,\ldots,a_{n-1},b_1,\ldots,b_m\ra$.
 Consider $r'=\la b_1,\ldots,b_m\ra$ and $r=r'\perp\la -a_n\ra$, then $p=q-r$.
 
 Consider the variety $X=Q\cup_{R'}R$ - the union of two
 quadrics $Q$ and $R$ intersecting at $R'$. It is a singular variety embedded into $\pp^N$. Since $Q$ is anisotropic, we have: $p\subset q$ $\Leftrightarrow$
 $R$ is isotropic $\Leftrightarrow$ $X$ is isotropic.
 Note that the intersection of $Q$ and $R$ is fixed, but the (multidimensional) angle between these two quadrics
 is not specified, we may vary it. This identifies $X$ as the base set of the linear system of quadrics in $\pp^N$:
 $$
 \la b_1,\ldots,b_m\ra\perp
 \begin{pmatrix}
  a_1 & 0 & . & 0 &\lambda_1 \\
  0 & a_2 & . & 0 & \lambda_2 \\
   . & . & . & . & . \\
  0 & 0 & . & a_{n-1} & \lambda_{n-1} \\
  \lambda_1 & \lambda_2 & . & \lambda_{n-1} & -a_n
 \end{pmatrix}
 $$
 Diagonalising it we get the linear system
 $q\perp -\la\text{ values of }p\ra$. The generic representative of it is $s=q\perp -\la\text{the generic value of } p\ra$. So, by the Theorem of Colliot-Th\'el\`ene and Levine, we get:
 $$
 p\subset q\hspace{2mm}\Leftrightarrow\hspace{2mm}
 X\text{ is isotropic }\Leftrightarrow\hspace{2mm}
 S\text{ is isotropic }\Leftrightarrow\hspace{2mm}
 q\text{ represents the generic value of }p.
 $$
 \Qed
\end{proof}

\section{Some consequences}
 \label{four}

\subsection{$2$-equivalence of field extensions}

The fact that pure symbols over the flexible closure of $k$ control the $2$-isotropy of projective 
varieties over $k$ implies that these also control the
$2$-equivalence classes of field extensions of $k$.
The {\it $p$-equivalence} of field extensions in characteristic zero was described in \cite[5.2]{INCHKm}. For a field of characteristic not $p$ we need to modify it a bit.

\begin{defi}
Let $E/k$ and $F/k$ be two field extensions of a field of $char\neq p$. Let 
$F=\op{colim} F_{\mu}$ of finitely generated extensions.
By Gabber's $l'$-altered desingularization \cite[Theorem 2.1]{I-T}, for any $\mu$, there exists a finite field extension $k'/k$ of degree prime to $p$ and a smooth
projective variety $P'_{\mu}$ over $k'$ with an embedding
$F_{\mu}\row k'(P'_{\mu})$ of $k$-extensions of finite
prime to $p$ degree. Then we say that $E/k\stackrel{p}{\geq}F/k$, if, for all $\mu$, $P'_{\mu}|_{E'}$ is $p$-isotropic, for all points $\op{Spec}(E')$ of 
$\op{Spec}(E\otimes_k k')$. We call extensions {\it $p$-equivalent}
$E/k\stackrel{p}{\sim}F/k$, if they are equal in the sense of this partial order.
\end{defi}

For a field extension $E/L$, denote as
$\op{P.P.}(\kker(E/L))$ the pure part of the kernel\\
$\kker(K^M_*(L)/2\row K^M_*(E)/2)$. We have:

\begin{prop}
 \label{two-equiv-ker}
 Let $char(k)\neq 2$ and $E/k$ and $F/k$ be field extensions. Then
 \begin{itemize}
  \item[$(1)$] $E/k\stackrel{2}{\geq}F/k$ $\Leftrightarrow$
$\op{P.P.}(\kker(\wt{F}/\wt{k}))\subset
\op{P.P.}(\kker(\wt{E}/\wt{k}))$;
  \item[$(2)$] $E/k\stackrel{2}{\sim}F/k$ $\Leftrightarrow$
$\op{P.P.}(\kker(\wt{F}/\wt{k}))=
\op{P.P.}(\kker(\wt{E}/\wt{k}))$.
 \end{itemize}
\end{prop}

\begin{proof}
 It is enough to prove (1). ($\Rightarrow$) If $u\in 
\op{P.P.}(\kker(\wt{F}/\wt{k}))$, where $F=\op{colim}F_{\mu}$, then, for some $\mu$, $u\in 
\op{P.P.}(\kker(\wt{F_{\mu}}/\wt{k}))$. 
By Gabber's $l'$-altered desingularization, we have an
extension $k'(P'_{\mu})/F_{\mu}$ of odd degree,
where $P'_{\mu}$ is a smooth projective variety over $k'$
and $k'$, in turn, is an extension of $k$ of odd degree. 
Our condition $E/k\stackrel{2}{\geq}F/k$ implies that $P'_{\mu}|_{E'}$ is $2$-isotropic, for any point $\op{Spec}(E')$ of $\op{Spec}(E\otimes_k k')$. Since $[k':k]$ is odd, among these points there will be some, for which $[E':E]$ is odd as well. As $u\in 
\op{P.P.}(\kker(\wt{F_{\mu}}/\wt{k}))$, it also belongs to
$\op{P.P.}(\kker(\wt{E'}(P'_{\mu})/\wt{k}))$, and since
$P'_{\mu}|_{E'}$ is $2$-isotropic, it belongs to
$\op{P.P.}(\kker(\wt{E'}/\wt{k}))$. Since $[E':E]$ is odd, we have: $u\in\op{P.P.}(\kker(\wt{E}/\wt{k}))$.

($\Leftarrow$) Let $F=\op{colim}F_{\mu}$.
As above, for each $\mu$, we have a finite extension
$k'/k$ of odd degree and a smooth variety $P'_{\mu}$ over $k'$ with an embedding of $k$-extensions $F_{\mu}\row k'(P'_{\mu})$ of odd degree. Let $P_{\mu}$ be the $k$-variety $P'_{\mu}\row\op{Spec}(k')\row\op{Spec}(k)$.
By Theorem \ref{Main}, there is a pure symbol $\alpha\in K^M_*(\wt{k})/2$, such that the $2$-isotropy of $P_{\mu}$ is equivalent to the 
triviality of $\alpha$. Since $P_{\mu}|_{k'(P'_{\mu})}$ is $2$-isotropic, so is $P_{\mu}|_F$. Thus,
$\alpha_{\wt{F}}=0$. By our condition, $\alpha_{\wt{E}}=0$ as well. Then $P_{\mu}|_{E}$ is $2$-isotropic.
This implies that $P'_{\mu}|_{E'}$ is $2$-isotropic, for any point $\op{Spec}(E')$ of $\op{Spec}(E\otimes_k k')$.
Thus, $E/k\stackrel{2}{\geq}F/k$.
 \Qed
\end{proof}

Over a field of characteristic zero, the $2$-equivalence classes of extensions of the base field parametrize the
{\it isotropic} points of characteristic $2$ of the Balmer spectrum $\op{Spc}(DM_{gm}(k))$ \cite{Bal-1} of the Voevodsky motivic category - see \cite[Theorem 5.13]{INCHKm}. 
The above result shows that these points may be distinguished with the help of pure symbols over the flexible closure of the base field.
Moreover, using our Main Theorem (for a field $k$ of $char=0$) one can show:

\begin{thm} {\rm (\cite[Theorem 3.7]{BSPS})}
All isotropic points ${\frak{a}}_{2,E}$ of characteristic $2$ of the Balmer spectrum $\op{Spc}(DM_{gm}(k))$ are
closed. In particular, there are no specialisation relations among such points. 
\end{thm}

\subsection{Numerical equivalence of cycles}

Our main result also implies that numerical triviality of cycles (mod 2) is controlled by symbols over flexible closure. Here we will assume that characteristic of our field is zero.

When restricted to a larger field, a numerically trivial cycle may stop to be numerically trivial
(as new cycles may be defined there). The following result permits to control this process.

\begin{thm}
 \label{num-triv-symb}
 Let $X$ be a smooth projective variety over $k$ and $u\in\op{Ch}^*(X)$ be a numerically trivial cycle $(mod\,2)$. Then there exists a purely transcendental extension $k(A)/k$ and
a pure symbol $\alpha\in K^M_*(k(A))/2$, such that, for all $L/k$,
$u_L$ is numerically trivial $\Leftrightarrow$ $\alpha_{L(A)}\neq 0$.
\end{thm}

\begin{proof}
We can identify $u$ with a map $T(d)[2d]\row M(X)$ in $\dmkF{\ff_2}$.  
By \cite[Theorem 4.12]{INCHKm}, there exists a purely transcendental extension $k'/k$
and a smooth projective anisotropic variety $Z'$ (of dimension $d$) over $k'$, such that $u_{k'}$ decomposes as $T(d)[2d]\row M(Z')\row M(X)_{k'}$. Moreover, the proof of \cite[Theorem 4.8]{INCHKm} shows that, if $F/k$ is such that $u_F$ is still numerically trivial, then
$Z'_{F'}$ is still anisotropic, where $F'=k'*_k F$. Conversely, if $Z'_{F'}$ is anisotropic, then 
the class $u_{F'}$ is anisotropic and thus, numerically trivial. Then so is $u_F$ (as the extension $F'/F$ is purely transcendental). Thus, $u_F$ is numerically trivial $\Leftrightarrow$ $Z'_{F'}$ is anisotropic. 

Now, by Theorem \ref{Main}, there exists a purely transcendental extension $k''/k'$ and
a pure symbol $\alpha\in K^M_*(k'')/2$, such that, for any field extension $L'/k'$, 
$Z'_{L'}$ is isotropic $\Leftrightarrow$ $\alpha_{L''}=0$, where $L''=L'*_{k'}k''$.
Denote $k''=k(A)$. Then, for any field extension $L/k$, 
$$
u_L\,\,\text{is numerically trivial}\,\,\Leftrightarrow Z'_{L'}\,\,\text{is anisotropic}\,\,
\Leftrightarrow\alpha_{L(A)}\neq 0.
$$
 \Qed
\end{proof}

Additionally, we can see that our numerically trivial class $u$ is killed when restricted to the
{\it reduced Rost motive} $\wt{M}_{\alpha}$ (see \cite[Theorem 3.5]{Iso}).

\begin{prop}
 \label{nt-rRm}
 Let $u\in\op{Ch}^*(X)$ be a numerically trivial class $(mod\, 2)$ and $\alpha\in K^M_*(k(A))/2$ be the respective pure symbol from Theorem \ref{num-triv-symb}.
Then $u_{k(A)}\otimes id_{\wt{M}_{\alpha}}=0$.
\end{prop}

\begin{proof}
 From the proof of Theorem \ref{num-triv-symb} we have purely transcendental extension
$k'/k$ and a smooth projective variety $Z'$ over $k'$, so that $u_{k'}$ factors as
$T(d)[2d]\row M(Z')\row M(X)_{k'}$. Since $Z'$ is isotropic over every point $\op{Spec}(L')$ of itself, $\alpha_{L''}$ is zero, for every such point (where $L''=L'*_{k'}k''$). Thus,
$M(Z')_{k''}\otimes\wt{M}_{\alpha}=0$ (note that the reduced Rost motive vanishes simultaneously with $\alpha$). Hence, $u_{k''}\otimes id_{\wt{M}_{\alpha}}=0$.
 \Qed
\end{proof}

From Theorem \ref{num-triv-symb} and Proposition \ref{nt-rRm} we immediately get:

\begin{cor}
 \label{Chow-num-triv-symb}
 Let $N\in Chow(k;\ff_2)$ be numerically trivial. Then there exists
a purely transcendental extension $k(A)/k$ and a pure symbol $\alpha\in K^M_*(k(A))/2$,
such that, for any field extension $L/k$, the motive $N_L$ is numerically trivial
$\Leftrightarrow$ $\alpha_{L(A)}\neq 0$. 
Moreover, $N_{k(A)}\otimes\wt{M}_{\alpha}=0$.
\end{cor}

\end{document}